\documentclass[11pt]{article}

\usepackage[T1]{fontenc}
\usepackage[letterpaper,margin=1.05in]{geometry}
\usepackage{amsmath,amssymb,amsthm}
\usepackage{microtype}
\usepackage[hidelinks]{hyperref}

\hypersetup{
  pdftitle={Uniform Testability Implies Asymmetric Testability},
  pdfauthor={Senhan Yao},
  pdfsubject={Hypothesis testing under filtrations},
  pdfkeywords={uniform consistency, asymmetric consistency, stationary ergodic processes, sequential decisions}
}
\newtheorem{theorem}{Theorem}[section]
\newtheorem{corollary}[theorem]{Corollary}
\theoremstyle{definition}
\newtheorem{definition}[theorem]{Definition}
\theoremstyle{remark}
\newtheorem{remark}[theorem]{Remark}

\newcommand{\N}{\mathbb{N}}
\newcommand{\E}{\mathcal{E}}
\newcommand{\F}{\mathcal{F}}

\title{Uniform Testability Implies Asymmetric Testability}
\author{Senhan Yao}
\date{August 8, 2026}

\begin{document}

\maketitle

\begin{abstract}
We prove a filtration-level conversion from uniform fixed-time testing to
all-sample one-sided error control and uniform future-tail control. Given an
adapted deterministic binary test whose worst-case fixed-time error tends to
zero, the construction selects update times with prescribed summable error
budgets and holds each selected decision until the next update. For every
prescribed significance level, the construction yields a test that controls
Type~I error at every sample size; from each update time onward, the probability
of any subsequent error is bounded uniformly over each hypothesis class by the
remaining budget. Hence the test is uniformly consistent and eventually correct
almost surely under both hypotheses. Applying the conversion to finite-alphabet
coordinate processes proves the first implication in Ryabko's Conjecture~5.2.
\end{abstract}

\section{Introduction}

Ryabko conjectured the chain
\[
  \text{uniform consistency}
  \quad\Longrightarrow\quad
  \text{asymmetric consistency}
  \quad\Longrightarrow\quad
  \text{asymptotic consistency}
\]
for tests of stationary ergodic processes
\cite[Conjecture~5.2, p.~71]{Ryabko2019}. This note establishes the first
implication; it does not address the second.

Related conversions are available under additional structure. For binary tests
under i.i.d.\ product sampling, Pfanzagl established that a uniformly weakly
consistent test exists if and only if a uniformly strongly consistent test
exists \cite{Pfanzagl1968}. For i.i.d.\ observations on a separable metric space,
Boeken et al.\ give a modern formulation and related equivalences among
uniform weak, future-tail (``strong''), and summable error control, and among
asymptotic and all-sample one-sided level guarantees
\cite[Definition~3 and Theorem~2, pp.~6--8]{BoekenEtAl2026}.
In Ermakov's terminology, distinguishability yields exponential worst-case
fixed-time errors and, by a union bound, exponential control of any future
error
\cite[equation~(2.6), Theorem~3.4, and equations~(3.2)--(3.3)]{Ermakov2017}.
For dependent data, Nobel uses a related update-and-hold construction whose
update rules are obtained from topological separation and uniform ergodic
estimates; summable bounds then imply almost-sure discernibility
\cite[Appendix~8.2, pp.~267--268]{Nobel2006}.

The contribution here is modest but structural. We make no priority claim
for the subsequence-and-summability or update-and-hold devices themselves;
related forms appear in the work cited above. What is isolated here is a
filtration-level conversion requiring only uniform fixed-time consistency.
Starting from any uniformly consistent adapted deterministic binary test,
uniform convergence itself supplies deterministic update times at which the
worst-case errors meet any prescribed summable budget. Holding each selected
decision until the next update then gives uniform future-tail control, while
initializing the held test at zero gives the all-sample one-sided level
constraint required for asymmetric consistency. Consequently, the
uniform-to-asymmetric implication in Conjecture~5.2 follows without
independence, product structure, topology, ergodic estimates, or
computability. The update times are existential and may be noncomputable,
consistently with the existence-based definitions considered here.

Section~2 gives the abstract framework, Section~3 proves the conversion, and
Section~4 applies it to finite-alphabet stationary ergodic processes.

\section{Framework and definitions}

Write \(\N:=\{1,2,\ldots\}\). Let
\((\Omega,\F_\infty)\) be a measurable space, and let
\((\F_n)_{n\ge1}\) be a filtration of sub-\(\sigma\)-fields of
\(\F_\infty\):
\[
  \F_1\subseteq\F_2\subseteq\cdots\subseteq\F_\infty.
\]
Let \(\mathcal P(\Omega,\F_\infty)\) denote the set of probability measures
on \((\Omega,\F_\infty)\), and let
\[
  H_0,H_1\subseteq\mathcal P(\Omega,\F_\infty).
\]
Either hypothesis class may be empty, and no behavior is prescribed
outside \(H_0\cup H_1\). We do not require
\(\F_\infty=\sigma(\bigcup_{n\ge1}\F_n)\).

Equip \(\{0,1\}\) with its discrete \(\sigma\)-field. A deterministic binary
test is a sequence
\[
  \eta=(\eta_n)_{n\ge1},
  \qquad
  \eta_n:\Omega\to\{0,1\},
\]
where \(\eta_n\) is \(\F_n\)-measurable. Equivalently, the sequence
\(\eta\) is \emph{adapted} to the filtration \((\F_n)_{n\ge1}\). Output \(i\)
is interpreted as selecting \(H_i\). Under \(\rho\in H_0\), the event
\(\{\eta_n=1\}\) is a
Type~I error; under \(\rho\in H_1\), the event \(\{\eta_n=0\}\) is a
Type~II error.

For \(i\in\{0,1\}\), write
\[
  \{\eta_n=i\text{ eventually}\}
  :=
  \bigcup_{N=1}^{\infty}
  \bigcap_{n=N}^{\infty}
  \{\eta_n=i\}.
\]
This event belongs to \(\F_\infty\), and, viewing
\(\{0,1\}\subset\mathbb{R}\), \(\eta_n\to i\) is equivalent to
\(\eta_n=i\) eventually.

\begin{definition}[Uniform consistency]\label{def:uniform}
A test \(\eta\) is \emph{uniformly consistent} for \(H_0\) against \(H_1\)
if, for every \(\varepsilon>0\), there is an \(N\in\N\) such that, for every
\(n\ge N\),
\[
  \rho(\eta_n\ne i)<\varepsilon
  \qquad
  (i\in\{0,1\},\ \rho\in H_i).
\]
The pair \((H_0,H_1)\) is \emph{uniformly testable} if such a test exists.
\end{definition}

\begin{remark}[Overlap of the hypotheses]\label{rem:overlap}
No disjointness assumption on \(H_0\) and \(H_1\) is needed in the
definitions. If a uniformly consistent test \(\eta\) exists, however, then
\(H_0\cap H_1=\varnothing\). Indeed, for
\(\rho\in H_0\cap H_1\) and every \(n\),
\[
  \rho(\eta_n\ne0)+\rho(\eta_n\ne1)=1,
\]
whereas Definition~\ref{def:uniform}, applied with
\(\varepsilon=1/2\), would eventually make both terms strictly smaller than
\(1/2\).
\end{remark}

For \(i\in\{0,1\}\), define the worst-case error
\[
  e_{i,n}(\eta):=
  \begin{cases}
    \displaystyle\sup_{\rho\in H_i}\rho(\eta_n\ne i),
      & H_i\ne\varnothing,\\[2mm]
    0, & H_i=\varnothing,
  \end{cases}
\]
and set
\[
  e_n(\eta):=\max_{i\in\{0,1\}}e_{i,n}(\eta).
\]
Definition~\ref{def:uniform} is equivalent to \(e_n(\eta)\to0\).
Indeed, suppose that \(\eta\) is uniformly consistent, and fix
\(\varepsilon>0\). Applying Definition~\ref{def:uniform} with
\(\varepsilon/2\), we obtain, for all sufficiently large \(n\),
\[
  e_{i,n}(\eta)\le \varepsilon/2<\varepsilon
  \qquad(i\in\{0,1\}),
\]
where the empty-class error is zero. Hence \(e_n(\eta)\to0\). Conversely,
\(e_n(\eta)\to0\) implies uniform consistency because
\(\rho(\eta_n\ne i)\le e_n(\eta)\) for every \(\rho\in H_i\).

\begin{definition}[Asymmetric consistency]\label{def:asymmetric}
For \(\alpha\in(0,1)\), a test
\(\psi^\alpha=(\psi_n^\alpha)_{n\ge1}\) is
\emph{asymmetrically consistent at level \(\alpha\)} if
\begin{align}
  \rho(\psi_n^\alpha=1)
  &\le \alpha
  &&(\rho\in H_0,\ n\in\N), \label{eq:level}\\
  \rho\bigl(\psi_n^\alpha=1\text{ eventually}\bigr)
  &=1
  &&(\rho\in H_1). \label{eq:power}
\end{align}
An \emph{asymmetrically consistent family} is a collection
\((\psi^\alpha)_{\alpha\in(0,1)}\) containing one such test for each level;
no compatibility across levels is required. The pair \((H_0,H_1)\) is
\emph{asymmetrically testable} if such a family exists.
\end{definition}

\begin{definition}[Asymptotic consistency]\label{def:asymptotic}
A test \(\theta\) is \emph{asymptotically consistent} if
\[
  \rho\bigl(\theta_n=i\text{ eventually}\bigr)=1
  \qquad
  (i\in\{0,1\},\ \rho\in H_i).
\]
\end{definition}

\section{Uniform-to-asymmetric conversion}

\begin{theorem}\label{thm:conversion}
Let \(\varphi=(\varphi_n)_{n\ge1}\) be a uniformly consistent deterministic
binary test for \(H_0\) against \(H_1\). Fix \(\alpha\in(0,1)\), and let
\(b=(b_k)_{k\ge1}\) satisfy
\[
  0<b_k\le\alpha
  \quad(k\in\N),
  \qquad
  \sum_{k=1}^{\infty}b_k<\infty.
\]
Then there exist a deterministic test
\(\psi=(\psi_n)_{n\ge1}\) and strictly increasing sample sizes
\[
  m_1<m_2<\cdots
\]
such that
\begin{equation}
  \rho(\psi_n=1)\le\alpha
  \qquad(\rho\in H_0,\ n\in\N),
  \label{eq:thm-level}
\end{equation}
and, for every \(i\in\{0,1\}\) with \(H_i\ne\varnothing\) and every
\(K\in\N\),
\begin{equation}
  \sup_{\rho\in H_i}
  \rho\left(\exists n\ge m_K:\ \psi_n\ne i\right)
  \le \sum_{k=K}^{\infty}b_k.
  \label{eq:general-tail}
\end{equation}
Consequently, \(\psi\) is uniformly and asymptotically consistent and is
asymmetrically consistent at level \(\alpha\).
\end{theorem}

\begin{proof}
Fix \(\alpha\) and \(b=(b_k)_{k\ge1}\), and write
\[
  e_n:=e_n(\varphi).
\]
Uniform consistency of \(\varphi\) gives \(e_n\to0\). We select a subsequence
on which the worst-case errors are summable and hold each selected decision
until the next update.

For each \(k\in\N\), because \(b_k>0\) and \(e_n\to0\), there is
an \(N_k\in\N\) such that
\[
  e_n\le b_k
  \qquad(n\ge N_k).
\]
Set \(m_0:=0\) and define recursively
\begin{equation}
  m_k:=\min\{n>m_{k-1}:e_n\le b_k\},
  \qquad k\in\N.
  \label{eq:selected-times}
\end{equation}
Indeed, any \(n>\max\{m_{k-1},N_k\}\) belongs to the set in
\eqref{eq:selected-times}, so that set is nonempty. The quantities \(e_n\) and \(b_k\)
are deterministic, so the update times \(m_k\) are deterministic as well.
Hence \((m_k)_{k\ge1}\) is strictly increasing, tends to infinity, and satisfies
\begin{equation}
  e_{m_k}\le b_k
  \qquad(k\in\N).
  \label{eq:selected-errors}
\end{equation}

For \(k\in\N\), let
\[
  I_k:=\{n\in\N:m_k\le n<m_{k+1}\},
\]
and define
\[
  \psi_n:=
  \begin{cases}
    0, & n<m_1,\\[1mm]
    \varphi_{m_k}, & n\in I_k.
  \end{cases}
\]
The blocks \((I_k)\) partition \(\{n\ge m_1\}\). The initial
constant-zero decisions enforce the Type~I constraint without affecting
eventual correctness. For \(n\in I_k\),
\(\F_{m_k}\subseteq\F_n\), so \(\psi_n\) is \(\F_n\)-measurable;
the case \(n<m_1\) is immediate. Thus \(\psi\) is a deterministic
test.

For \(\rho\in H_0\), the Type~I error is zero when \(n<m_1\). If
\(n\in I_k\), then
\[
  \rho(\psi_n=1)
  =\rho(\varphi_{m_k}=1)
  \le e_{0,m_k}(\varphi)
  \le e_{m_k}
  \le b_k
  \le\alpha.
\]
Thus \eqref{eq:thm-level}, and therefore the level condition
\eqref{eq:level}, holds at every sample size.

For \(i\in\{0,1\}\) and \(K\in\N\), let
\[
  A_{i,K}:=
  \left\{\exists n\ge m_K:\ \psi_n\ne i\right\}.
\]
This event belongs to \(\F_\infty\), being a countable union of measurable
error events. Since every \(n\ge m_K\) lies in a unique block \(I_k\) with
\(k\ge K\),
\begin{equation}
  A_{i,K}
  =\bigcup_{k=K}^{\infty}\{\varphi_{m_k}\ne i\}.
  \label{eq:tail-identity}
\end{equation}
For \(\rho\in H_i\), countable subadditivity and
\eqref{eq:selected-errors} give
\begin{align*}
  \rho(A_{i,K})
  &\le\sum_{k=K}^{\infty}\rho(\varphi_{m_k}\ne i)\\
  &\le\sum_{k=K}^{\infty}b_k.
\end{align*}
Taking the supremum over a nonempty \(H_i\) proves
\eqref{eq:general-tail}.

Let \(B_K:=\sum_{k=K}^{\infty}b_k\). For \(n\ge m_K\),
\(\{\psi_n\ne i\}\subseteq A_{i,K}\), so
\[
  e_{i,n}(\psi)\le B_K
\]
for nonempty \(H_i\), while the empty-class error is zero. Given
\(\varepsilon>0\), choose \(K\) such that \(B_K<\varepsilon\). Then, for
every \(n\ge m_K\) and \(i\in\{0,1\}\),
\(e_{i,n}(\psi)<\varepsilon\). Hence \(e_n(\psi)\to0\), so \(\psi\) is
uniformly consistent.

Fix \(i\in\{0,1\}\) and \(\rho\in H_i\). The events \(A_{i,K}\)
decrease with \(K\). Since \(m_K\to\infty\), the tails
\(\{n\ge m_K\}\) are cofinal among the integer tails, so
\[
  \bigcap_{K=1}^{\infty}A_{i,K}
  =
  \{\psi_n\ne i\text{ infinitely often}\}.
\]
Since \(\rho(A_{i,K})\le B_K\to0\), continuity from above gives
\[
  \rho(\psi_n\ne i\text{ infinitely often})=0.
\]
Hence \(\psi_n=i\) eventually almost surely. Thus \(\psi\) is asymptotically
consistent and, together with \eqref{eq:thm-level}, asymmetrically consistent
at level \(\alpha\).

\end{proof}

\begin{corollary}[Uniform testability implies asymmetric testability]
\label{cor:uniform-asymmetric}
If \((H_0,H_1)\) is uniformly testable, then, for every
\(\alpha\in(0,1)\), there exist a deterministic test \(\psi^\alpha\) and
strictly increasing sample sizes \((m_k^\alpha)_{k\ge1}\) such that
\eqref{eq:level} holds and
\begin{equation}
  \sup_{\rho\in H_i}
  \rho\left(\exists n\ge m_K^\alpha:\ \psi_n^\alpha\ne i\right)
  \le \alpha\,2^{-K}
  \label{eq:tail}
\end{equation}
for every \(i\in\{0,1\}\) with \(H_i\ne\varnothing\) and every
\(K\in\N\). Consequently, \((H_0,H_1)\) is asymmetrically testable, and each
\(\psi^\alpha\) is also uniformly and asymptotically consistent.
\end{corollary}

\begin{proof}
Let \(\varphi\) be a uniformly consistent test. For each
\(\alpha\in(0,1)\), set
\[
  b_k^\alpha:=\alpha\,2^{-(k+1)},
  \qquad
  b^\alpha:=(b_k^\alpha)_{k\ge1}.
\]
Then \(\sum_{k=K}^{\infty}b_k^\alpha=\alpha\,2^{-K}\). Apply
Theorem~\ref{thm:conversion}, and denote the resulting test and
update times by \(\psi^\alpha\) and \((m_k^\alpha)_{k\ge1}\), respectively.
\end{proof}

\begin{remark}[Quantitative but existential tail control]
\label{rem:tail-strength}
The bound \eqref{eq:general-tail} is stronger than almost-sure eventual
correctness. Its right-hand side is explicit, but the update times are
generally existential: selecting \(m_k\) requires access to the worst-case
errors \(e_n(\varphi)\), which need not be computable from a finite
description of the testing problem.
\end{remark}

\section{Application to Ryabko's Conjecture}

Let \(A\) be a nonempty finite set, write
\(A^\ast:=\bigcup_{n=1}^{\infty}A^n\), let \(\Omega=A^\N\) carry the product
\(\sigma\)-field, and define
\[
  X_j(x_1,x_2,\ldots):=x_j,
  \qquad
  \F_n:=\sigma(X_1,\ldots,X_n),
  \qquad
  \F_\infty:=\sigma(X_1,X_2,\ldots).
\]
Let \(T(x_1,x_2,\ldots)=(x_2,x_3,\ldots)\), and let \(\E\) denote the
class of stationary ergodic probability measures \(\rho\) on
\((A^\N,\F_\infty)\), that is, those satisfying
\(\rho\circ T^{-1}=\rho\) and
\[
  \rho(B\mathbin{\triangle}T^{-1}B)=0
  \quad\Longrightarrow\quad
  \rho(B)\in\{0,1\}
  \qquad(B\in\F_\infty).
\]

For the prefix map
\[
  \pi_n:A^\N\to A^n,
  \qquad
  \pi_n(x_1,x_2,\ldots)=(x_1,\ldots,x_n),
\]
the fibers of \(\pi_n\) are the atoms of \(\F_n\). Hence every adapted
binary decision has the unique form \(\eta_n=f_n\circ\pi_n\) for a function
\(f_n:A^n\to\{0,1\}\), and conversely every such function defines an
adapted decision. Fix \(a_\ast\in A\). In particular, for each
\(\alpha\in(0,1)\), the test \(\psi^\alpha\) considered below corresponds to
the finite-prefix test \(g_n^\alpha:A^n\to\{0,1\}\) defined by
\[
  g_n^\alpha(x_1^n)
  :=\psi_n^\alpha(x_1,\ldots,x_n,a_\ast,a_\ast,\ldots).
\]
Because \(\psi_n^\alpha\) is \(\F_n\)-measurable, its value is constant on
each fiber of \(\pi_n\). Hence the definition is independent of the chosen
continuation beyond \(x_n\), and
\(\psi_n^\alpha=g_n^\alpha\circ\pi_n\). Thus defining
\(g^\alpha|_{A^n}:=g_n^\alpha\) for \(n\ge1\) gives a single
\(g^\alpha:A^\ast\to\{0,1\}\) in Ryabko's sense. Accordingly, the tests and error
conventions above are exactly Ryabko's finite-prefix tests
\cite[pp.~51--54]{Ryabko2019}.

Ryabko's uniform, asymmetric, and asymptotic criteria coincide with
Definitions~\ref{def:uniform}, \ref{def:asymmetric}, and \ref{def:asymptotic}, respectively
\cite[Definitions~5.1--5.3, pp.~53--54]{Ryabko2019}.\footnote{The prose immediately preceding Definition~5.2 says that the
Type~II error probability ``goes to 0,'' whereas item~(ii) of the formal
definition requires Type~II errors to occur only finitely many times almost
surely. We follow the formal Definition~5.2. Under the weaker fixed-time
formulation, uniform consistency already yields the asymmetric condition by
outputting \(0\) before a sufficiently late deterministic time and then using
the original test; the construction above addresses the stronger formal
almost-sure requirement.}
\footnote{Definition~5.3
is headed ``asymptotic consistency'' but contains the phrase ``uniformly
consistent.'' The displayed formula and surrounding discussion indicate that
this appears to be a typographical slip and that ``asymptotically
consistent'' is intended.} For binary decisions, convergence
to \(i\) is equivalent to eventual equality with \(i\).

This construction is admissible under Ryabko's definitions: a test is a
function of the observed finite prefix, so it may hold a decision between
updates, while the asymmetric criterion asks separately for existence at each
\(\alpha\) and imposes neither cross-level compatibility nor computability.
Our update times are deterministic, not data-dependent stopping times, the
test is defined at every \(n\), and the initial zeros are valid prefix
decisions. Thus a possibly noncomputable update schedule satisfies every
formal requirement of Ryabko's asymmetric consistency definition.

\begin{corollary}[Uniform-to-asymmetric implication in Ryabko's
Conjecture~5.2]
\label{cor:ryabko}
Let \(H_0,H_1\subseteq\E\). If \((H_0,H_1)\) is uniformly testable, then it
is asymmetrically testable. More precisely, for every \(\alpha\in(0,1)\)
there is a test \(\psi^\alpha\) that is asymmetrically consistent at level
\(\alpha\), uniformly consistent, and asymptotically consistent, and that
satisfies the uniform tail-error bound \eqref{eq:tail}.
\end{corollary}

\begin{proof}
Apply Corollary~\ref{cor:uniform-asymmetric} to the coordinate filtration.
\end{proof}


\begin{thebibliography}{99}

\bibitem{BoekenEtAl2026}
Philip Boeken, Eduardo Skapinakis, Konstantin Genin, and Joris M. Mooij,
``Topological criteria for hypothesis testing with finite-precision
measurements,''
arXiv:2601.13946v2 [math.ST], revised May 5, 2026.
\href{https://arxiv.org/abs/2601.13946v2}{arXiv:2601.13946v2}.

\bibitem{Ermakov2017}
Mikhail Ermakov,
``On consistent hypothesis testing,''
\emph{Journal of Mathematical Sciences} \textbf{225} (2017), 751--769;
author preprint, arXiv:1403.6296v6 [math.ST], April 20, 2015.
\href{https://arxiv.org/abs/1403.6296v6}{arXiv:1403.6296v6};
\href{https://doi.org/10.1007/s10958-017-3491-4}
     {doi:10.1007/s10958-017-3491-4}.

\bibitem{Nobel2006}
Andrew B. Nobel,
``Hypothesis testing for families of ergodic processes,''
\emph{Bernoulli} \textbf{12} (2006), no.~2, 251--269.
\href{https://doi.org/10.3150/bj/1145993974}
     {doi:10.3150/bj/1145993974}.

\bibitem{Pfanzagl1968}
Johann Pfanzagl,
``On the existence of consistent estimates and tests,''
\emph{Z. Wahrscheinlichkeitstheorie verw. Gebiete}
\textbf{10} (1968), no.~1, 43--62.
\href{https://doi.org/10.1007/BF00572921}
     {doi:10.1007/BF00572921}.

\bibitem{Ryabko2019}
Daniil Ryabko,
\emph{Asymptotic Nonparametric Statistical Analysis of Stationary Time
Series}, SpringerBriefs in Computer Science, Springer, Cham, 2019.
\href{https://doi.org/10.1007/978-3-030-12564-6}
     {doi:10.1007/978-3-030-12564-6}.
Author manuscript (used for the cited pagination):
\href{https://arxiv.org/abs/1904.00173v1}{arXiv:1904.00173v1}.


\end{thebibliography}
\end{document}